\documentclass[11pt]{article}

\usepackage[margin=1.0in,tmargin=0.75in,bmargin=0.85in]{geometry}

\usepackage{amsmath}
\usepackage{amssymb}
\usepackage{graphicx}
\usepackage{amsthm}
\usepackage{tikz} 
\usetikzlibrary{decorations.pathmorphing, decorations.pathreplacing}
\usepackage[pdftex,pagebackref=false]{hyperref} 
\usepackage[noabbrev]{cleveref}
\usepackage{lineno}
\usepackage{caption,subcaption}
\usepackage{pgfplots}
\pgfplotsset{compat=1.18}
\usetikzlibrary{shapes.geometric}

\newtheorem{theorem}{Theorem}
\newtheorem{problem}[theorem]{Problem}
\newtheorem{lemma}[theorem]{Lemma}
\newtheorem{definition}[theorem]{Definition}

\newtheorem{assumption}[theorem]{Assumption}

\newcommand{\R}{\mathbb{R}}
\newcommand{\Rn}{\R^n}

\begin{document}

\title{A Convex Optimization Approach to the Discrete Hanging Chain Problem}
\author{
{Russell Gabrys}\thanks{Mathematics Department, Walla Walla University}
\and
{Stefan Sremac$^*$}
}
\maketitle

\begin{abstract}
In this paper we investigate the discrete version of the classical hanging chain problem.  We generalize the problem, by allowing for arbitrary mass and length of each link.  We show that the shape of the chain can be obtained by solving a convex optimization problem.  Then we use optimality conditions to show that the problem can be further reduced to solving a single non-linear equation, when the links of the chain have symmetric mass and length.
\end{abstract}


\section{Introduction}
\label{sec:intro}

The hanging chain problem is to determine the shape of a chain that is attached at each end to a level beam and allowed to hang freely.  The earliest record of this classical problem is generally attributed da Vinci in the 16th century, \cite{Buk2008, Bar2004}.  The problem was solved in~1691, independently by Huygens, Leibniz, and Johann Bernoulli using geometry and techniques of the newly developed calculus of variations, \cite{Bar2004}.  In these solutions it is assumed that the chain is really a continuous cable, with uniform mass.  This assumption leads to an elegant, closed-form solution using hyperbolic functions.  The resulting curve is referred to as a \emph{catenary}.

The discrete version of the problem assumes that the chain is composed of $n$ links.  In the most simple version, the links have the same mass and length.  Unlike the continuous version, the discrete problem lacks a closed-form solution and relies on numerical techniques to determine the position of each link.  The problem may be formulated as a constrained continuous optimization problem on $n$ variables, see for instance Section 10.4 of \cite{Luen:84}.  Numerical optimization solvers are then used to obtain a solution.  In~\cite{Wang:1993}, the problem is reduced to a system of non-linear equations, while in~\cite{AgmYiz:2020}, an additional torque equilibrium condition is used to reduce the problem to solving a single non-linear equation.

In \cite{GriVan2005}, the authors consider several generalizations of the continuous cable problem and formulate each variation as a continuous optimization problem. The authors view their contribution as a tutorial and case study in optimization, stressing the importance of modeling and convex formulations.  The resulting optimization problems are solved numerically. 

We consider a more general version of the discrete problem, where the links need not have identical length or mass.  In Section~\ref{sec:form} we formulate the problem as a constrained optimization problem.  In Section~\ref{sec:conv} we show that the optimization problem may be reformulated as a convex optimization problem.  In Section~\ref{sec:symmetric} we prove that the chain hangs symmetrically and show that the optimal solution can be obtained by solving a single non-linear equation, as long as the links are symmetric. We view this contribution, similarly to that of \cite{GriVan2005}, as partly tutorial in nature, emphasizing the power of convex reformulation and analysis.

\section{Formulation of the Problem}
\label{sec:form}

We begin by formally stating the problem.
\begin{problem}
A chain with $n \ge 2$ links having masses $m_1, \dotso, m_n$ and lengths $\ell_1, \dotso, \ell_m$, respectiely, is suspended from a horizontal beam so that the first and last links are attached to the beam a distance~$d$ apart. Determine the shape that the free-hanging chain takes.
\label{prob1}
\end{problem}
The guiding physics principle is that the chain will take the shape that minimizes potential energy.  We follow Section 10.4 of \cite{Luen:84} to set up the optimization problem.  As in~Figure~\ref{fig:link}, let each link have unit length and let $x_i$ and $y_i$ denote the horizontal and vertical distances spanned by link~$i$, respectively.  
\begin{figure}[h]
\label{fig:main}
\begin{subfigure}{0.5\textwidth}
\centering
\includegraphics[width=0.9\textwidth]{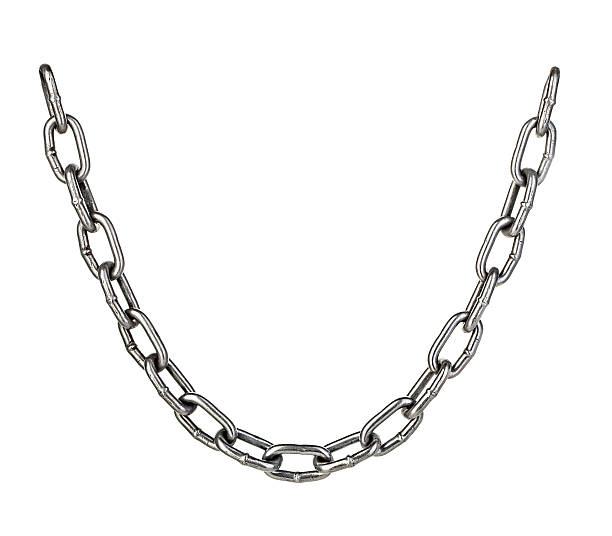}
\caption{Hanging chain with $n$=19 links.}
\end{subfigure}
\begin{subfigure}{0.5\textwidth}
\centering
\begin{tikzpicture}
\begin{scope}[rotate=45]
    \draw (0,0) ellipse (1.5 and 4);
    \draw (0,0) ellipse (0.75 and 3);
\end{scope}
\draw[dotted] (-2.1213,2.1213) -- node[left](){$y_i$}(-2.1213,-2.1213);
\draw[dotted] (-2.1213,-2.1213) -- node[below](){$x_i$}(2.1213,-2.1213);
\draw[dotted] (-2.1213,2.1213) --node[above](){$\ell_i$}(2.1213,-2.1213);
\end{tikzpicture}
\caption{A single link in the chain.}
\label{fig:link}
\end{subfigure}
\end{figure}
We view the chain from left to right so that $x_i\ge 0$ but $y_i$ may be positive, negative, or zero, corresponding to the slope of the link.  The link in~Figure~\ref{fig:link}, for instance, corresponds to~${y_i < 0}$.  If we assume that the center of mass for each link is in the middle of the link, then the potential energy of the first two links is,
\begin{align*}
PE_1 &= m_1g\left( h + \frac 12 y_1 \right) \text{ and } PE_2 = m_2g\left( h + y_1 + \frac 12 y_2 \right),
\end{align*}
where $h$ denotes the height of the beam and $g$ denotes the force due to gravity.  More generally we have,
\[
PE_i = m_ig\left( h + y_1 + \cdots + y_{i-1} + \frac 12 y_i \right),\quad  i \in \{1,\dotso,n\}.
\]
The potential energy of the chain, then, is the sum of the potential energies of the individual links, 
\begin{align*}
PE &= \sum_{i=1}^n PE_i = \sum_{i=1}^n m_ig\left( h + y_1 + \cdots + y_{i-1} + \frac 12 y_i \right) \\
&= gh\sum_{i=1}^n m_i + g(\frac 12 m_1 + m_2 + \cdots + m_n)y_1 +  g(\frac 12 m_2 + m_3 + \cdots + m_n)y_2 + \\
&   \qquad \qquad \cdots + g(\frac 12 m_{n-1} + m_n)y_{n-1} + g(\frac 12 m_n)y_n.
\end{align*}
For simplicity, we let $c_i$ denoted the coefficient of $y_i$, excluding $g$.  That is, for each $i \in \{1,\dotso, n\}$,
\begin{equation}
\begin{split}
c_i := \begin{cases}
\frac 12 m_i + \sum_{j=i+1}^n m_j \quad &\text{if } i < n,\\
\frac 12 m_n &\text{if } i=n.
\end{cases}
\end{split}
\label{eq:ci}
\end{equation}
Then the potential energy of the entire chain is,
\[
PE = gh\sum_{i=1}^n m_i + g \sum_{i=1}^n c_iy_i.
\]
The arrangement that minimizes the potential energy is the solution to the optimization problem,
\begin{equation*}
\min_{y \in \Rn} \left \{ \left( gh\sum_{i=1}^n m_i + g \sum_{i=1}^n c_iy_i \right) \ : \ \sum_{i=1}^ny_i = 0, \ \sum_{i=1}^n \sqrt{\ell_i^2-y_i^2} = d \right \}.
\end{equation*}
The first constraint requires the chain to start and end at the same height and the second constraint ensures that the chain spans the required distance $d$. 

We remove the constant term $gh\sum_{i=1}^n m_i$ and the coefficient $g$ from the objective, since they do not affect the optimal choice of $y$.  These changes yield the simplified optimization problem,
\begin{equation}
\min_{y \in \Rn} \left \{ \sum_{i=1}^n c_iy_i \ : \ \sum_{i=1}^ny_i = 0, \ d - \sum_{i=1}^n \sqrt{\ell_i^2-y_i^2} = 0 \right \} .
\label{eq:main}
\end{equation}
Note that we have rearranged the second constraint for reasons that will be apparent in the upcoming section.  

\section{Convex Formulation}
\label{sec:conv}

Clearly there exists a minimum value for optimization problem \eqref{eq:main}, since the objective is continuous and the feasible region is compact (closed by equality constraints and bounded since~$\lvert y_i \rvert \le \ell_i$). In this section we show that \eqref{eq:main} is equivalent to the optimization problem,
\begin{equation}
\min_{y\in \Rn} \left \{ \sum_{i=1}^n c_iy_i  \ : \ \sum_{i=1}^ny_i = 0, \ d - \sum_{i=1}^n \sqrt{\ell_i^2-y_i^2} \le 0 \right \}.
\label{eq:conv}
\end{equation}
Note that problem \eqref{eq:conv} is an instance of \emph{convex optimization} since the objective is a convex function, the equality constraint is linear, and the inequality constraint is also a convex function.  See, for instance, Chapter 8 of \cite{Beck:2023}.  In fact, the inequality constraint function
\begin{equation}
g(y) := d - \sum_{i=1}^n \sqrt{\ell_i^2-y_i^2},
\label{eq:gfun}
\end{equation}
is \emph{strictly convex}, since it is a sum of single variable functions of the form~${h(x) = - \sqrt{a^2 -x^2}}$, which are strictly convex.  We begin by making the following `common sense' assumptions.

\begin{assumption}
\label{assump:main}
We assume that $m_1,\dotso, m_n$ and $\ell_1,\dotso, \ell_n$ are positive and that $d \in ( \max_i \ell_i, \sum_i \ell_i )$.
\end{assumption}
The bounds on $d$ ensure that it is possible for the chain to span the given distance $d$ and that no single link exceeds the spanning length.  We now state the main result of this section.

\begin{theorem}
\label{thm:conv}
It holds that 
\begin{enumerate}
\item the optimization problem \eqref{eq:conv} is an instance of convex optimization,
\label{itm:conv1}
\item the set of optimal solutions to \eqref{eq:conv} is exactly the set of optimal solutions to \eqref{eq:main} and this set is a singleton.
\label{itm:conv2}
\end{enumerate}
\end{theorem}

\begin{proof}
Item~\ref{itm:conv1} holds by the arguments above and the definition of convex optimization. For item~\ref{itm:conv2}, it suffices to show that the optimal solution occurs when the inequality constraint is satisfied with equality.  To this end, let $\bar y$ be feasible for problem~\eqref{eq:conv} and suppose that the inequality constraint is satisfied strictly at $\bar y$.  Then,
\begin{equation}
 \sum_{i=1}^n \sqrt{\ell_i^2-\bar y_i^2} > d.
\label{eq:sum}
\end{equation}
By Assumption~\ref{assump:main} it holds that for each $i\in \{1,\dotso,n\}$ we have
\[
\sqrt{\ell_i^2-\bar y_i^2} \le \ell_i \le \max_i \ell_i < d.
\]  
It follows that at least two terms in the sum of \eqref{eq:sum} are positive.  Let~${\alpha, \beta \in \{1,\dotso,n\}}$ with $\alpha < \beta$ such that $\bar y_\alpha \in (-\ell_\alpha, \ell_\alpha)$ and $\bar y_\beta \in (-\ell_\beta, \ell_\beta)$. Then there exists~${\varepsilon >0}$ such that $ \hat y \in \Rn$ defined as, 
\[
\hat y_i = \begin{cases}
\bar y_i - \varepsilon \quad &\text{if } i=\alpha, \\
\bar y_i + \varepsilon &\text{if } i=\beta, \\
\bar y_i  &\text{otherwise},
\end{cases}
\]
satisfies the inequality constraint strictly, i.e., $g(\hat y) < 0$. Moreover, our new vector $\hat y$ satisfies the equality constraint since,
\[
\sum_{i=1}^n \hat y_i = \left(\sum_{i \ne \alpha, \beta} \bar y_i\right) + \bar y_\alpha - \varepsilon + \bar y_{\beta} + \varepsilon = \sum_{i=1}^n \bar y_i = 0.
\]
We have shown that $\hat y$ is feasible for \eqref{eq:conv}.  Now we show that $\hat y$ has a lower objective value than $\bar y$.  Indeed,
\begin{align*}
\sum_{i=1}^n c_i \bar y_i - \sum_{i=1}^nc_i \hat y_i &= c_\alpha \bar y_\alpha +c_\beta \bar y_\beta  - c_\alpha(\bar y_\alpha - \varepsilon) - c_\beta(\bar y_\beta + \varepsilon)\\
&= \varepsilon (c_\alpha - c_\beta) > 0.
\end{align*}
The inequality follows from the observation that $c_1 > \cdots > c_n$ from \eqref{eq:ci} and the choice of~${\alpha < \beta}$.  We have shown that $\bar y$ is not optimal. It follows that the optimal solutions of~\eqref{eq:conv} satisfy the inequality constraint tightly.  Therefore the set of optimal solutions for~\eqref{eq:conv} is exactly the same as the set of optimal solutions for~\eqref{eq:main}.

Lastly, to show that the set of optimal solutions is a singleton, suppose $\bar y$ and $\hat y$ are two distinct optimal solutions to \eqref{eq:conv}.  Since this problem is convex, the set of optimal solutions is convex, therefore,~${\tilde y = (\bar y + \hat y)/2}$ is also an optimal solution and distinct from $\bar y$ and $\hat y$.  Moreover, as shown above, the constraint function $g(y)$ as in \eqref{eq:gfun} is strictly convex.  Therefore,
\[
g(\tilde y) < \frac 12 g(\bar y) + \frac 12 g(\tilde y) = 0,
\]
a contradiction to the first part of item~\ref{itm:conv2}.
\end{proof}

\section{Solutions to the Symmetric Case}
\label{sec:symmetric}

Now that we have formulated the problem as a convex optimization problem, we use the Karush-Kuhn-Tucker (KKT) optimality conditions to reduce the optimization problem to solving a single non-linear equation.  Throughout this section we make the additional assumption that the links of the chain are symmetric.  That is, the first link is identical to the $n$th link, the second link is identical to the $(n-1)$th link, and so on.

\begin{assumption}
\label{assump:symm}
We assume that $m_1,\dotso, m_n$ and $\ell_1,\dotso, \ell_n$ are positive, that $d \in ( \max_i \ell_i, \sum_i \ell_i )$, and that
\[
m_i = m_{n+1-i} \text{ and } \ell_i = \ell_{n+1-i}, \quad i \in \{1,\dotso, \left \lfloor \frac n2 \right \rfloor\}.
\]
\end{assumption}
Since \eqref{eq:conv} is convex, the KKT conditions are sufficient, e.g., Theorem~11.12 of \cite{Beck:2023}.  Therefore, for any $(y, \lambda, \mu) \in \Rn \times \R \times \R$ that solve the system of equations below,
\begin{align}
c_i + \lambda + \mu \frac{y_i}{\sqrt{\ell_i^2-y_i^2}}&= 0, \quad i\in \{1,\dotso,n\} 
\label{eq:kkt1} \\
\sum y_i &= 0,
\label{eq:kkt2}\\
d - \sum \sqrt{\ell_i^2-y_i^2} &\le 0 ,
\label{eq:kkt3}\\
\mu &\ge 0,
\label{eq:kkt4}\\
\mu\left( d - \sum \sqrt{\ell_i^2-y_i^2} \right) &= 0.
\label{eq:kkt5}
\end{align}
it holds that $y$ is a solution to \eqref{eq:conv}.  By Theorem~\ref{eq:conv} we know that \eqref{eq:kkt3} is satisfied with equality at optimality.  Therefore, \eqref{eq:kkt5} is satisfied regardless of $\mu$.  Then the system reduces to \eqref{eq:kkt1} - \eqref{eq:kkt4}. Isolating \eqref{eq:kkt1} for $y_i$ we get
\begin{align*}
c_i + \lambda + \mu \frac{y_i}{\sqrt{\ell_i^2 - y_i^2}} = 0 &\implies (c_i + \lambda) \sqrt{\ell_i^2 - y_i^2} = -\mu y_i\\
&\implies (c_i + \lambda)^2 (\ell_i^2 - y_i^2) = (\mu y_i)^2\\
&\implies \ell_i^2 (c_i+\lambda)^2 = y_i^2 ((c_i + \lambda)^2 + \mu^2),
\end{align*}
which simplifies to,
\begin{equation}
y_i = - \frac{\ell_i (c_i + \lambda)}{\sqrt{(c_i + \lambda)^2 + \mu^2}}.
\label{eq:yi}
\end{equation}
The choice of negative sign for $y_i$ is justified by rewriting \eqref{eq:kkt1} as,
\[
\mu \frac{y_i}{\sqrt{\ell_i^2 - y_i^2}} = - (c_i + \lambda),
\]
and observing that the sign of $y_i$ is the same as the sing of $-(c_i + \lambda)$.  Note also, that \eqref{eq:yi} is well defined since $\mu >0$.  Indeed, if $\mu = 0$ then \eqref{eq:kkt1} implies that $c_i = -\lambda$ for all~${i \in \{1,\dotso,n\}}$. Consequently, $m_i = 0$ for $i \in \{1,\dotso,n-1\}$, violating Assumption~\ref{assump:symm}.  

We could substitute \eqref{eq:yi} into \eqref{eq:kkt2} and \eqref{eq:kkt3} to set up a system of nonlinear equations in $\lambda$ and $\mu$.  However, \eqref{eq:yi} may be simplified significantly by considering symmetry.  Intuitively, it makes sense that the chain should hang symmetrically if the links are symmetric.  That is, the position of the first link is opposite that of the $n$th link, and so on.  We define this formally in the following.

\begin{definition}
A vector $y \in \Rn$ that is feasible for \eqref{eq:conv} is said to be symmetric if 
\[
y_i = -y_{n+1-i} \quad  i \in \{ 1,\dotso,\left \lfloor \frac n2 \right \rfloor \},
\]
and $y_{(n+1)/2} = 0$ when $n$ is odd.
\label{def:sym}
\end{definition}

Now, assuming a symmetric chain as in Assumption~\ref{assump:symm}, a symmetric solution as in Definition~\ref{def:sym}, and the case where $n$ is odd we obtain an explicit expression for $\lambda$.  Indeed, considering \eqref{eq:yi} in this context, we have
\begin{align*}
y_{(n+1)/2} = 0 &\implies -\frac{\ell_{(n+1)/2}(c_{(n+1)/2}+\lambda)}{\sqrt{(c_{(n+1)/2} - \lambda)^2 + \mu^2}} = 0 \\
&\implies \ell_{(n+1)/2}(c_{(n+1)/2}+\lambda) = 0\\
&\implies \lambda = -c_{(n+1)/2}. 
\end{align*}
Substituting back into \eqref{eq:yi} we get,
\begin{equation}
y_i = - \frac{\ell_i(c_i - c_{(n+1)/2})}{\sqrt{(c_i - c_{(n+1)/2})^2+\mu^2}}, \qquad i \in \{1,\dotso,n\}.
\label{eq:yiodd}
\end{equation}
Before we prove that this solution is actually optimal, let us first consider the even case.  There is no simple way to derive $\lambda$ in the even case, so we take a cue from the odd case.  Note that $c_{(n+1)/2} = \frac 12 m_{(n+1)/2} + m_{(n+1)/2 + 1} + \cdots + m_n$ is exactly half the mass of the chain.  We propose that~$\lambda$ equal to the negative of half the mass of the chain is suitable in the even case as well.  To this end we consider the solution,
\begin{equation}
y_i = - \frac{\ell_i(c_i - \bar c)}{\sqrt{(c_i - \bar c)^2+\mu^2}}, \qquad i\in \{1,\dotso,n\},
\label{eq:yifinal}
\end{equation}
where $\bar c = \frac 12 (m_1 + \cdots + m_n)$.  We proceed by proving the following technical result.
\begin{lemma}
\label{lem:tech}
Let $\bar c = \frac 12 (m_1 + \cdots + m_n)$. For each $i \in \{1,\dotso,n\}$ it holds that
\[
c_i - \bar c = -(c_{n+1-i} - \bar c).
\]
\end{lemma}
\begin{proof}
Let $i \in \{2,\dotso,n-1\}$. Expanding the left hand side we have,
\begin{align*}
c_i - \bar c &= \frac 12 m_i + m_{i+1} + \cdots + m_n - \frac 12 \left( m_1 + \cdots + m_n\right) \\
&= -\frac 12\left( m_1 + \cdots + m_{i-1}\right) + \frac 12 \left(m_{i+1} + \cdots + m_n \right) .
\end{align*}
Expanding the right hand side we have,
\begin{align*}
-(c_{n+1-i} - \bar c) &= -\left( \frac 12 m_{n+1-i} + m_{n+2-i} + \cdots + m_n - \frac 12 \left( m_1 + \cdots + m_n\right) \right) \\
&= -\frac 12 (m_{n+2-i} + \cdots + m_n) + \frac 12 (m_1 + \cdots + m_{n-i})\\
&= -\frac 12\left( m_1 + \cdots + m_{i-1}\right) + \frac 12 \left(m_{i+1} + \cdots + m_n \right),
\end{align*}
where the last step follows from $m_i = m_{n+1-i}$ as in Assumption~\ref{assump:symm}.  The case $i \in \{1,n\}$ follows analogously.
\end{proof}
We now prove the main result.
\begin{theorem}
Let $\bar c = \frac 12 (m_1 + \cdots + m_n)$.  Then there exists unique $\mu > 0$ such that $y\in \Rn$ defined as 
\[
y_i = -\frac{\ell_i(c_i - \bar c)}{\sqrt{(c_i - \bar c)^2+\mu^2}}, \quad i\in \{1,\dotso,n\},
\]
is the unique optimal solution to \eqref{eq:conv} and is symmetric.
\label{thm:symopt}
\end{theorem}
\begin{proof}
First we show that $y$ as constructed above is symmetric according to Definition~\ref{def:sym}.  To this end let $i \in \{1,\dotso,n\}$.  Then using Lemma~\ref{lem:tech} and $\ell_i = \ell_{n+1-i}$ as in Assumption~\ref{assump:symm} we have,
\[
y_{n+1-i} = -\frac{\ell_{n+1-i}(c_{n+1-i} - \bar c)}{\sqrt{(c_{n+1-i} - \bar c)^2+\mu^2}} = \frac{\ell_{i}(c_i - \bar c)}{\sqrt{(c_i - \bar c)^2+\mu^2}} = -y_i.
\]
Next we show that $y$ satisfies equations \eqref{eq:kkt1}, \eqref{eq:kkt2}, \eqref{eq:kkt3}, and \eqref{eq:kkt4}.  First, note that \eqref{eq:kkt1} is satisfied by construction and \eqref{eq:kkt2} follows from the fact that $y$ is symmetric.  Now we consider \eqref{eq:kkt3}.  Expanding and simplifying we have,
\begin{align*}
d - \left( \sum_{i=1}^n \sqrt{\ell_i^2 - y_i^2}\right) &= d- \left( \sum_{i=1}^n \sqrt{\ell_i^2 - \frac{\ell_i^2(c_i - \bar c)^2}{(c_i -\bar c)^2+\mu^2}}\right)\\
&= d- \left( \sum \frac{\ell_i \mu}{\sqrt{(c_i -\bar c)^2+\mu^2}}\right)\\
&= d- \mu \left( \sum \frac{\ell_i}{\sqrt{(c_i -\bar c)^2+\mu^2}}\right).
\end{align*}
We now define $\phi : \R \to \R$ as,
\[
\phi(\mu) = d- \mu \left( \sum \frac{\ell_i}{\sqrt{(c_i -\bar c)^2+\mu^2}}\right).
\]
Then \eqref{eq:kkt3} is equivalent to $\phi(\mu) = 0$.  Note that $\phi$ is defined and continuous on $(0, \infty)$, we just need to show that $\phi$ admits a root on $(0, \infty)$. To this end,
\begin{align*}
\lim_{\mu \to \infty} \phi(\mu) &= d - \lim_{\mu \to \infty} \mu \left( \sum \frac{\ell_i}{\sqrt{(c_i -\bar c)^2+\mu^2}}\right)\\
&= d - \left( \sum_{i=1}^n \lim_{\mu \to \infty} \frac{\ell_i}{\sqrt{1+(\frac{c_i -\bar c}{\mu})^2}}\right)\\
&= d - \left( \sum_{i=1}^n \ell_i \right) < 0.
\end{align*}
The inequality is due to the assumed bounds on $d$ in Assumption~\ref{assump:symm}.  Next, we consider the behavior of $\phi$ near $0$.  Note that if $c_i - \bar c \ne 0$ for all $i \in \{1,\dotso,n\}$ then $\phi(0)=d>0$.  On the other hand, if there exists an index $j \in \{1,\dotso,n\}$ such that $c_j - \bar c = 0$ we have that $\phi(0)$ is not defined.  Therefore, we consider the limit.  Note that there can be at most one such index since $c_1 > \cdots > c_n$.  We have,
\begin{align*}
\lim_{\mu \searrow 0} \phi(\mu) &= d - \lim_{\mu \searrow 0}\left( \sum_{i=1}^n \frac{\mu \ell_i}{\sqrt{(c_i -\bar c)^2+\mu^2}}\right)\\
&= d - \lim_{\mu \searrow 0} \left( \sum_{i \ne j}  \frac{\mu \ell_i}{\sqrt{(c_i -\bar c)^2+\mu^2}}\right) - \lim_{\mu \searrow 0} \left( \frac{\mu \ell_j}{\sqrt{\mu^2}}\right)\\
&= d - \ell_j > 0.
\end{align*}
Here also, the inequality follows from Assumption~\ref{assump:symm}.  We have shown that $\lim_{\mu \to \infty} \phi(\mu) < 0$ and that $\lim_{\mu \searrow 0} \phi(\mu) > 0$.  It follows from the intermediate value theorem that there exists a positive~$\mu$ satisfying equation \eqref{eq:kkt3}.  Note that we have also satisfied \eqref{eq:kkt4}. 

Finally, the uniqueness of $y$ follows from Theorem~\ref{thm:conv} and the uniqueness of $\mu$ is due to the monotonicity of $\phi$.  Indeed,
\begin{align*}
\phi ' (\mu) &= \sum_{i=1}^n \frac{-\ell_i \sqrt{(c_i - \bar c)^2 +\mu^2} + \frac{\mu^2 \ell_i}{\sqrt{(c_i -\bar c)^2 +\mu^2}}}{(c_i -\bar c)^2 +\mu^2}\\
&=  \sum_{i=1}^n - \frac{\ell_i (c_i - \bar c)^2}{\left((c_i - \bar c)^2 +\mu^2\right)^{3/2}}.
\end{align*}
Clearly, $\phi'$ is negative, implying the desired result.
\end{proof}

We have shown that solving the discrete hanging chain problem reduces to finding a root of the univariate function $\phi$. This task is significantly aided by the fact that $\phi$ is monotonically decreasing, as shown above.  Moreover, $\phi$ is convex on $(0, \infty)$, since the second derivative,
\[
\phi '' (\mu) = \sum_{i=1}^n \frac{3\ell_i(c_i-\bar c)^2 \mu}{((c_i - \bar c)^2 + \mu^2)^{5/2}},
\]
is positive whenever $\mu > 0$.  It is well known that Newton's method, for instance, performs very well on monotone convex functions.

\bibliographystyle{plain}
\bibliography{mybib}

\end{document}